\documentclass[fms,times]{cuparticle_v2}

\usepackage{latexsym,amssymb}
\usepackage{boldline}
\usepackage{algorithm}
\usepackage[noend]{algpseudocode}
\usepackage{amsmath,amssymb,amsfonts,amsthm,color,mathtools,cuted}
\usepackage{enumerate}
\usepackage{eucal}
\usepackage{environ}
\usepackage{graphicx}
\usepackage{amsthm}
\usepackage{booktabs}
\usepackage{tikz,subcaption}
\usepackage{multicol}
\newtheorem{theorem}{Theorem}
\newdefinition{definition}{Definition}
\usepackage{verbatim}

\usetikzlibrary{positioning}
\usetikzlibrary{arrows,shapes}
\usepackage{natbib} 
\usepackage{epstopdf}
\usepackage{epsfig} 
\usepackage{psfrag}
\usepackage[width=.85\textwidth]{caption}
\usepackage{lipsum}
\usepackage{url}
\usepackage[width=.85\textwidth, font=footnotesize]{caption}

\renewcommand{\cite}{\citep}

\newtheorem{lemma}{Lemma}
\newtheorem{corollary}{Corollary}

\newtheorem{proposition}{Proposition}

\DeclareRobustCommand{\stirling}{\genfrac\{\}{0pt}{}}
\algdef{SE}[SUBALG]{Indent}{EndIndent}{}{\algorithmicend\ }%
\algtext*{Indent}
\algtext*{EndIndent}
\makeatletter
\renewcommand\bibsection%
{
  \section*{\refname
    \@mkboth{\MakeUppercase{\refname}}{\MakeUppercase{\refname}}}
}
\makeatother

\newcommand{\ignore}[1]{}

\DeclareMathAlphabet{\mathcal}{OMS}{cmsy}{m}{n}      

\volume{}
\doi{}

\begin{document}

\authorheadline{Blake Wilson}
\runningtitle{Bounds on sweep-covers by Raney Numbers}
 
\begin{frontmatter}

\title{Bounds on sweep-covers by Raney Numbers}
\author[1]{Blake Wilson}
 
\address[1]{School of Electrical and Computer Engineering, Purdue University, West Lafayette, IN
  \ead{wilso692@purdue.edu}}

\received{10 January 2022}
 
\begin{abstract}
In this work, we introduce a vertex separator in trees known as a sweep-cover that is defined by an ancestor-descendent relationship with all nodes in the tree. We prove the recurrence relation of sweep-covers with $n$ subcovers $P_{\Delta, \gamma}(n)$ on a class of infinite $\Delta$-ary trees with constant path lengths $\gamma$ between the $\Delta$-star internal nodes. Then, we provide recurrence relations for Raney numbers over integer compositions and show that they provide a lower-bound for sweep-covers such that $P_{\Delta, \gamma}(n) = \Omega\left( \frac{\sqrt{2 \pi} n^{\Delta n + \Delta + \frac{3}{2}}}{e^n ((\Delta-1)n+\Delta+1)!(n+1)!} \gamma \right)$.
\end{abstract}

\end{frontmatter}

\section{Introduction}
\label{sec:Introduction}
Much of the early work in combinatorial graph theory was devoted to computing the enumeration functions of classes of trees with respect to the number of nodes or edges in the tree \cite{Harary1964,Otter1948}. However, there exist structures within trees whose generating function can only be defined for a particular class of trees. One such structure is a vertex separator\cite{10.5555/500866,Escalante1972}. A vertex separator $S$ is a collection of nodes in a graph $\mathcal{T}$ that disconnects two sets of nodes from one another. More specifically, a sweep-cover is a collection of sets of sibling nodes in a tree that disconnects the root node, $v_s \in \mathcal{T}$, from all the leaf nodes $\mathcal{G} \subseteq \mathcal{T}$ such that no two nodes in $S$ are an ancestor or descendant of one another. We also allow for the collection $S$ to contain the root node or a subset of leaf nodes. However, due to the ancestor-descendant requirement for sweep-covers, a sweep-cover that contains the root node cannot also contain a subset of leaf nodes. In this work, we provide an algorithm for computing all sweep-covers in a given tree. Then, we prove the recurrence relation for sweep-covers on infinite, labeled and directed (ILD) trees. Finally, we show a surjective relationship between the set of Raney trees and the set of sweep-covers in ILD trees.
\section{Sweep-Covers}
\label{sec:sweep_covers}
This section provides the formal definition of a sweep-cover and proves how to compute all sweep-covers in a tree using the later defined $\alpha$ and $\pi$ mappings. Let $\mathcal{T}= \{\mathcal{V}, \mathcal{E}\}$ be a directed tree where $\mathcal{V}$ and $\mathcal{E}$ are the sets of nodes and edges, respectively, in $\mathcal{T}$. Let $v_s \in \mathcal{V}$ be the root node of $\mathcal{T}$. Given any tree, a collection of nodes in the tree that satisfies the following definition is considered a sweep-cover.

\begin{definition}[Sweep-Cover]
Consider a directed tree $\mathcal{T} = \{\mathcal{V}_\mathcal{T}, \mathcal{E}_\mathcal{T} \}$. Let $S \subset 2^{\mathcal{V}_\mathcal{T}}$ be a collection of non-empty sets of nodes chosen from $\mathcal{V}_\mathcal{T}$, i.e., $\forall S_i \in S : S_i \subset \mathcal{V}_\mathcal{T}$. We will refer to each set in a sweep-cover as a subcover, i.e., $S_i \in S$. Let $\mathcal{C}$ be the set of all nodes in $S$, i.e., $\mathcal{C} = \bigcup\limits_{S_i \in S} S_i$. Then, we denote the sets $\mathcal{A}_\mathcal{C}$ and $\mathcal{D}_\mathcal{C}$ to be the set of all ancestors and the set of all descendants, respectively, of the nodes in $\mathcal{S}$. We say $S$ is a sweep-cover of $\mathcal{T}$ if it fulfills the following conditions:

\begin{enumerate}
    \item Each distinct pair $S_i,S_j \in S$, where $S_i \ne S_j$, is disjoint i.e., $S_i \cap S_j = \emptyset$. \label{sweep_cover:disjointne}
    \item Each set may only contain sibling nodes, i.e. nodes which share a parent. \label{sweep_cover:sibling}
    \item Let $\mathcal{C} = \bigcup\limits_{S_i \in S} S_i$, then $\mathcal{C} \cup \mathcal{A}_{\mathcal{C}} \cup \mathcal{D}_{\mathcal{C}} = \mathcal{V}_\mathcal{T}$. \label{sweep_cover:complete_graph}
    \item For any unique pair of nodes $v_i,v_j \in \bigcup\limits_{S_i \in S} S_i$, $v_i$ cannot be a descendant or ancestor of $v_j$.  \label{sweep_cover:ancestor}
\end{enumerate}
\label{def:SweepCover}
\end{definition}

\subsection{Constructing Sweep-Covers in Trees}

Let's begin with the simplest example of a sweep-cover; $S = \{\{v_s\}\}$ where $v_s$ is the root node of the tree of interest. We refer to this sweep-cover as the \textit{root} sweep-cover. One can verify that the root sweep-cover satisfies Definition \ref{def:SweepCover} for any tree. Beginning here, we can show that every other sweep-cover can be defined via a sequence of mappings from this original sweep-cover. For example, consider the set of children $C$ of the root node. The first mapping for constructing sweep-covers is to swap a singleton set $\{v\} \in S$ with the set of its children $C$, i.e., $S \backslash \{v\} \cup C$. This not only applies to sweep-covers that contain one singleton, but any number of singletons in a sweep-cover can be replaced by their children. We define this swap using the following mapping.

\begin{definition}[Parental Map - $\alpha$]
Consider a sweep-cover $S$ for a tree $\mathcal{T}$. We require that $S$ contains at least one singleton subcover $\{v\}$ and $v$ has a nonempty set of children $C$. Then, $\alpha$ maps a sweep-cover $S$ to another sweep-cover by replacing $\{v\}$ with $C$, i.e., $\alpha(S,v) = S \backslash \{v\} \cup C$. Likewise, we define the inverse of $\alpha$ as $\alpha^{-1}(S,C) = S \backslash C \cup \{\{v\}\}$ when $C$ contains a complete set of sibling nodes and their parent exists, i.e., $C$ is not a singleton of the root node. \label{def:alpha_function}
\end{definition}

\begin{lemma}
$\alpha(S,v)$ and $\alpha^{-1}(S,C)$ always produce a sweep-cover.
\label{lem:alpha_func}
\end{lemma}
\begin{proof}
 Let $S$, $\mathcal{T}$, $v$, and $C$ be defined as in Definition \ref{def:alpha_function} and let $S' = \alpha(S,v)$. First, $C$ naturally satisfies Definitions \ref{def:SweepCover}.\ref{sweep_cover:disjointne} and \ref{def:SweepCover}.\ref{sweep_cover:sibling}. By definition, $\mathcal{A}_{S'} = \mathcal{A}_{S} \cup \{v\}$ and $\mathcal{D}_{S'} = \mathcal{D}_{S} \backslash C$. Therefore,
\begin{align}
    \bigcup\limits_{S_i \in S \backslash \{v\}} S_i \cup (C \backslash \{v\}) \cup (\mathcal{A}_{S} \cup \{v\} ) \cup (\mathcal{D}_{S} \backslash C) = \bigcup\limits_{S_i \in S} S_i \cup \mathcal{A}_{S} \cup \mathcal{D}_{S}  = \mathcal{V}, \label{eq:sc_comp_graph_proof}
\end{align}
which implies $S'$ satisfies Definition \ref{def:SweepCover}.\ref{sweep_cover:complete_graph}. Likewise, if $S$ satisfied Definition \ref{def:SweepCover}.\ref{sweep_cover:ancestor}, then no other nodes in $S$ are ancestors or descendants of $\{v\}$. Because the nodes in $C$ share all ancestors and descendants with $v$, no other nodes in $S'$ are ancestors or descendants of $C$. Therefore, $S'$ must satisfy Definition \ref{def:SweepCover}.\ref{sweep_cover:ancestor}. By following the arguments above and inverting the roles of $C$ and $\{v\}$, one can show that $\alpha^{-1}(S,C)$ will always produce a sweep-cover if $C$ has a parent $\{v\}$.
\end{proof}

Continuing the previous example, let the sweep-cover $S$ be composed of the children of the root node of a tree, i.e., $S = \{C\}$. Suppose $C$ is composed of two nodes, i.e., $C = \{v_1, v_2\}$. If we partition $C$ into two sets $C_1 = \{v_1\}$ and $C_2 = \{v_2\}$ and redefine $S$ to be $S = \{C_1, C_2\}$, then we can verify that $S$ is still a sweep-cover. This kind of partitioning extends to any sweep-cover with a subcover of more than one node. In fact, partitioning subcovers into smaller subcovers is the second mapping we use to construct sweep-covers from other sweep-covers. To do this, we first define the set of partitions of a set $C$ as $P_C$.

\begin{definition}[Partition Map - $\pi$]
Consider a sweep-cover $S$ for a tree $\mathcal{T}$. Assume that $S$ contains a non-singleton subcover $S_i \in S$. Let $p \in P_{S_i}$ be a partition of $S_i$. We define the $\pi$-mapping to take a sweep-cover $S$ and construct a new sweep-cover by replacing $S_i$ with all the sets in $p$, i.e., $\pi(S,p) = S \backslash S_i \cup p$. Likewise, we can also define the inverse of $\pi$ as $\pi^{-1}(S,S_i) = S \backslash p \cup S_i$. Note that due to the uniqueness of nodes in $\mathcal{T}$, it is sufficient to only specify $S$ and $S_i$ (or $p$) for arguments in both $\pi$ and its inverse. \label{def:pi_function}
\end{definition}

\begin{corollary}
By noting that the union of any partition of a set is the original set, one can follow similar arguments to those presented in Lemma \ref{lem:alpha_func} to show that $\pi(S,p)$ and $\pi^{-1}(S,S_i)$ produce a sweep-cover.
\label{cor:partition_function}
\end{corollary}

\begin{definition} [Depth]
Let $\mathcal{T}= \{\mathcal{V}, \mathcal{E}\}$ be a directed tree, with a root node, $v_s \in \mathcal{V}$.  Define the {\it depth} of a given node $v \in \mathcal{V}$, to be the number of edges in the unique path from $v_s$ to $v$ in the tree. Assume that the depth of the root node $v_s$ is 0.
\label{def:depth}
\end{definition}

\begin{definition} [Linear Path]
Consider a set of edges $\mathcal{E}$ in some tree $\mathcal{T}$. We define a {\it linear path} to be an ordered set of edges $\{ (v_{i_1}, v_{i_2}), (v_{i_2},v_{i_3})  ... (v_{i_{m-1}},v_{i_m}) \} \subset \mathcal{E}$ that connects the initial node $v_{i_1}$ to the terminal node $v_{i_m}$ such that each edge's source node in the linear path has at most one outgoing edge in $\mathcal{E}$.
\label{def:Linear_Path}
\end{definition}

To see how we can combine both $\alpha$ and $\pi$ mappings, first consider the root sweep-cover $S = \{\{v_s\}\}$. Beginning with the root node, as long as the subcover of interest contains a node which is in a linear path, then we can only apply $\alpha$-mappings until the subcover is a singleton leaf node or it is replaced with a non-singleton, i.e., $S = \{C\}$. Then, we can apply a $\pi$-mapping to construct a sweep-cover $S = \{S_1,S_2,...,S_n\}$ where every subcover is a part of $C$. Let's consider a simple  partition of $C$ into only two sets $p = \{S_1,S_2\}$ such that $S_1 = \{v_1\}$ and $S_2 = \{v_2, ..., v_{|C|} \}$. Then, apply $\pi(S,p)$ to redefine $S$ to be $S = \{S_1,S_2\}$. Noting that $S_1$ is a singleton, we can immediately apply $\alpha(S,S_1)$ to replace it with a set of the children $C_1$ of $v_1$, i.e., $S = \{C_1,S_2\}$. This process of using $\pi$-mappings to create singletons and then using $\alpha$-mappings to use nodes of greater depth continues until $S$ contains only leaf nodes and the mappings can no longer be applied. If we begin with the root sweep-cover and enumerate all possible ways of applying $\alpha$ and $\pi$ mappings in this manner, then a large set of sweep-covers can be constructed. But, is this process sufficient to construct all sweep-covers? As it turns out, the answer is yes. We can show this by contradiction. To do this, first we will show that all sweep-covers, except the root sweep-cover, are constructible from a previous sweep-cover by applying $\alpha^{-1}$ or $\pi^{-1}$. Then, we will show that recursively applying the inverse mappings $\alpha^{-1}$ and $\pi^{-1}$ on any sweep-cover will always construct the root sweep-cover.

\begin{lemma}
Consider a sweep-cover $S$ for a tree $\mathcal{T}$ with a root node $v_s \in \mathcal{T}$. If $S$ is not the root sweep-cover, then $S$ is constructible from another sweep-cover $S'$ by applying $\alpha$ or $\pi$ mappings on $S'$. \label{lem:inverting_mapping}
\end{lemma}
\begin{proof}
We will prove this by contradiction. Suppose there exists a sweep-cover $S = \{S_1,S_2,...,S_n \}$ which is not constructible from $\alpha$ or $\pi$. Then, either it is impossible to invert $\alpha$ or $\pi$ on $S$ or to invert either mapping does not yield a sweep-cover. We know by Lemma \ref{lem:alpha_func} and Corollary \ref{cor:partition_function} that $\alpha^{-1}$ and $\pi^{-1}$ produce sweep-covers if $S$ is a sweep-cover. So, it must be that we cannot apply $\alpha^{-1}$ and $\pi^{-1}$ to $S$. Assuming it is impossible to apply $\alpha^{-1}$, then there does not exist a subcover $S_i \in S$ that contains a complete set of sibling nodes for $\alpha^{-1}$. If it is also impossible to apply $\pi^{-1}$, then to satisfy Definitions \ref{def:SweepCover}.\ref{sweep_cover:complete_graph} and \ref{def:SweepCover}.\ref{sweep_cover:ancestor}, there must exist a subset of subcovers $\bar{S} \subset S$ that contains nodes which are descendants of the siblings of the nodes in $S_i$. But, to ensure that $\alpha^{-1}$ and $\pi^{-1}$ cannot be applied on $\bar{S}$, we must recursively apply the same arguments to $\bar{S}$ ad infinitum. Eventually, we will need to consider a subset of subcovers that contain only leaf nodes. However, to ensure that Definitions \ref{def:SweepCover}.\ref{sweep_cover:complete_graph} and \ref{def:SweepCover}.\ref{sweep_cover:ancestor} are satisfied on a set of leaf nodes, either all the leaf nodes are contained in one subcover, in which case $\alpha^{-1}$ can be applied to the subcover, or there exists a collection of subcovers which partition the set of leaf nodes, in which case $\pi^{-1}$ can be applied to the collection. Therefore, for every sweep-cover, it must be the case that one can apply $\alpha^{-1}$ or $\pi^{-1}$, which will always yield another sweep-cover if $S$ is not the root sweep-cover.
\end{proof}

Applying Lemma \ref{lem:inverting_mapping}, one can easily see that applying a sequence of $\alpha^{-1}$ and $\pi^{-1}$ mappings beginning at any sweep-cover $S$ will eventually yield the root sweep-cover.

\begin{lemma}
The set of all sweep-covers in a tree are constructible by beginning with the root sweep-cover and applying a sequence of $\alpha$ and $\pi$ mappings. \label{lem:all_sweep_construct_ap}
\end{lemma}
\begin{proof}
By Lemma \ref{lem:inverting_mapping}, we can always apply $\alpha^{-1}$ and $\pi^{-1}$ on a sweep-cover that isn't the root sweep-cover. Note that by using $\alpha^{-1}$, we replace a set of sibling nodes for a singleton of their parent node. This means that a sequence of $\alpha^{-1}$ mappings will always replace subcovers with ancestors of a lower depth. Likewise, $\pi^{-1}$ consolidates a partition of sibling nodes into a single set of sibling nodes, which can be immediately inverted to a singleton of their parent node via $\alpha^{-1}$. So, a sequence of $\alpha^{-1}$ and $\pi^{-1}$ will repeatedly replace subcovers with their parent until a node does not have a parent. Due to the in-degree of a tree being one, the root node is an ancestor of every node in the tree. Because the only node that doesn't have a parent is the root node, the final sweep-cover after applying an exhaustive sequence of $\alpha^{-1}$ and $\pi^{-1}$ mappings must yield the root sweep-cover.
\end{proof}

Now that we know all sweep-covers are constructible from sequences of $\alpha$ and $\pi$ mappings, we will analyze the mapping $h : \mathcal{H} \rightarrow 2^{\mathcal{V}_{\mathcal{T}}}$ between the set of all sequences of $\alpha$ and $\pi$ mappings $\mathcal{H}$ to sweep-covers. Firstly, we know the mapping $h$ is non-injective and surjective. Assuming the domain $\mathcal{H}$ is well constructed so that each sequence does not violate the assumptions of Definitions \ref{def:alpha_function} and \ref{def:pi_function}, then every sequence of $\alpha$ and $\pi$ mappings constructs one sweep-cover, i.e., $h$ is surjective. Additionally, consider the fact that two sequential set partitions can be equivalent to one set partition. Likewise, two sequential $\pi$-mappings can be equivalent to one $\pi$-mapping. So, multiple sequences can yield the same sweep-cover, i.e., $h$ is non-injective.

\begin{corollary}
The mapping $h : \mathcal{H} \rightarrow 2^{\mathcal{V}_{\mathcal{T}}}$ between the set of all sequences of $\alpha$ and $\pi$ mappings $\mathcal{H}$ to the set of all sweep-covers is non-injective and surjective.
\end{corollary}

\begin{figure}[ht]
    \centering
    \includegraphics[width=0.25\linewidth]{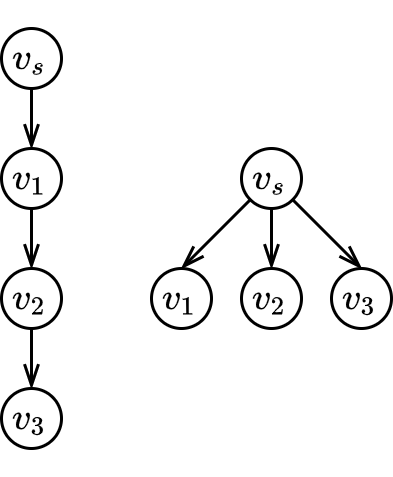}
    \caption{Both trees contain four nodes and three edges. The left tree has only 4 sweep-covers, all of which are constructible from $\alpha$-mappings, whereas the right tree has 6 sweep-covers which are constructible from $\alpha$ and $\pi$ mappings.}
    \label{fig:different_sweep_covers}
\end{figure}

Secondly, we know that the sequence of $\alpha$ and $\pi$ mappings depend entirely on the tree structure. For example, consider the trees in Figure \ref{fig:different_sweep_covers}. Despite both trees having the same number of nodes and edges in Figure \ref{fig:different_sweep_covers}, each tree has a different number of sweep-covers and set of sequences of $\alpha$ and $\pi$ mappings for constructing their sweep-covers. Therefore, we cannot construct a function that computes them given only the number of nodes or edges in a tree without knowing the tree's structure. However, it is still possible to compute all sweep-covers using an algorithm. But, to introduce such an algorithm, we will first consider a recursive formalism of the $\alpha$ and $\pi$ mappings from the perspective of the nodes in the tree. Here, we restrict the algorithm to only generate sweep-covers with $n$ subcovers, though one could modify it to generate all number of sweep-covers up to an $n$. First, we must observe that sweep-covers can be defined by sets of sweep-covers on subtrees. Namely, given a sweep-cover $S$, there exist subsets of subcovers in $S$ which form sweep-covers of subtrees in $\mathcal{T}$. In fact, all subcovers in a sweep-cover belong to a sweep-cover for a subtree in the original tree. Combining all of these subtree sweep-covers constructs an overall sweep-cover in the tree. To prove this, we will introduce another crucial aspect of sweep-covers, their induced subtrees.

\subsection{Induced Subtrees From Sweep-Covers}

\begin{definition}[Induced Subtrees from Sweep-Covers]
Consider a sweep-cover $S$ on a tree $\mathcal{T} = \{ \mathcal{V}, \mathcal{E} \}$. For each subcover $S_i \in S$, construct a set of nodes $\mathcal{V}_i$ by taking every node in $S_i$ along with all of its ancestors and descendants, i.e., $\mathcal{V}_i = S_i \cup \mathcal{A}_{S_i} \cup \mathcal{D}_{S_i}$.Next, define $\mathcal{E}_i$ to be the set of all edges in $\mathcal{E}$ which can be constructed from the nodes in $\mathcal{V}_i$, i.e., $\mathcal{E}_i = \{ (v_i,v_j) \in \mathcal{E} : v_i,v_j \in \mathcal{V}_i \}$. Using these sets, we construct an induced subtree $\mathcal{B}_i = \{ \mathcal{V}_i, \mathcal{E}_i \}$ for each subcover. We say the set $\mathcal{B} = \{\mathcal{B}_i : S_i \in S \}$ is a set of induced subtrees defined by $S$. \label{def:induced_subtree}
\end{definition}
By a simple extension of Definition \ref{def:SweepCover}.\ref{sweep_cover:complete_graph} with Definition \ref{def:induced_subtree}, $\bigcup\limits_{\mathcal{T}_i \in \mathcal{B}}\mathcal{V}_i = \mathcal{V}$ and $\bigcup\limits_{\mathcal{T}_i \in \mathcal{B}}\mathcal{E}_i = \mathcal{E}$.
\begin{corollary}
The collection of induced subtrees $\mathcal{B}$ in a tree $\mathcal{T}$ satisfies $\bigcup\limits_{\mathcal{B}_i \in \mathcal{B}}\mathcal{B}_i = \mathcal{T}$.  \label{cor:sweep_cover_preserves_R}
\end{corollary}
\begin{definition}[Lowest Linear Descendant]
Consider a tree $\mathcal{T} = \{\mathcal{V}, \mathcal{E} \}$. For each node $v \in \mathcal{V}$, define $\mathcal{L}(v)$ to be the descendant of $v$ (or $v$ itself) such that a linear path exists in $\mathcal{E}$ from $v$ to $\mathcal{L}(v)$ and $\mathcal{L}(v)$ has an out degree not equal to 1. 
\label{def:lowest_linear_descendant}
\end{definition}

\begin{lemma}
 Consider a sweep-cover $S$ on a tree $\mathcal{T}$ with a set of induced subtrees $\mathcal{B}$ defined by a subcover $S_i$ from a tree with root node $v_s \in \mathcal{V}$. Then, there exists a linear path between $v_s$ and $v_p$ in $\mathcal{B}_i$. If $S_i$ contains multiple nodes, then $v_p$ is the lowest linear descendant of $v_s$ in $\mathcal{B}_i$.
 \label{lem:sweep_cover_linear_path}
\end{lemma}
\begin{proof}
By definition of $\mathcal{E}_i$, there do not exist edges to nodes outside of those that exist between the nodes in $\mathcal{V}_i$. Because there only exists one incoming edge from the parent of all the ancestors of $v_p$, withholding the root node, then there exists a linear path between $v_s$ and $v_p$ by Definition \ref{def:Linear_Path}. If $S_i$ contains more than one node, then $v_p$ has an out-degree greater than one so $v_p$ is the lowest linear descendant of $v_s$. 
\end{proof}

\begin{corollary}
 Consider a sweep-cover $S$ on a tree $\mathcal{T}$ with a set of induced subtrees $\mathcal{B}$. For each induced subtree $\mathcal{B}_i \in \mathcal{B}$ and its corresponding subcover $S_i \in S$, the collection $\{S_i\}$ is a sweep-cover for $\mathcal{B}_i$. \label{cor:induced_subtree_covered}
\end{corollary}
\begin{proof}
Lemma \ref{lem:sweep_cover_linear_path} implies there exists a linear path between $v_s$ and the parent of the nodes in $S_i$ in $\mathcal{T}_i$. Therefore, one can repeatedly apply $\alpha$-mappings from $v_s$ to $S_i$ and show that $\{S_i\}$ is a sweep-cover with one subcover for $\mathcal{B}_i$.
\end{proof}
\begin{lemma}
 Consider a sweep-cover $S$ on a tree $\mathcal{T}$. Given a node $v_u \in \mathcal{T}$, let $\mathcal{T}_u$ be the subtree induced by $v_u$ that contains all of the descendants and ancestors of $v_u$. If there exists a subcover $S_i$ which contains descendants of $v_u$, then there exists a subset of subcovers $S'_u \subseteq S$, that form a sweep-cover of $\mathcal{T}_u$.
 \label{lem:induced_subtree_covered}
\end{lemma}
\begin{proof}
To ensure that Definition \ref{def:SweepCover}.\ref{sweep_cover:complete_graph} is satisfied without violating Definition \ref{def:SweepCover}.\ref{sweep_cover:ancestor}, there must exist a set of subcovers $S'_u$ with nodes in $\mathcal{T}_u$ which are not ancestors of $S_i$ but whose ancestors and descendants contain all of $\mathcal{T}_u$. Naturally, if $S$ is a sweep-cover, then Definitions \ref{def:SweepCover}.\ref{sweep_cover:disjointne} and \ref{def:SweepCover}.\ref{sweep_cover:sibling} are satisfied for $S'_u$ as well.
\end{proof}

We can apply this result to construct all sweep-covers recursively by computing sweep-covers on induced subtrees rooted at different nodes. First, we will show that given an oracle which can compute all sweep-covers of size $m < n$, we can compute all sweep-covers of size $n$. This will be the main proof for the induction step for our algorithm in the following section. But, before we can prove this, we need to introduce our notation for integer compositions to assist the remainder of this work.

\begin{definition}[Integer Compositions]
We denote $\Gamma_{x,m}$ to be the collection of all integer compositions of the integer $x$ into $m$ parts, i.e., there are $m$ ordered integers in each partition set which sum to $x$. For a given ordered integer partition of $x = x_1 + x_2 + ... + x_m$, the corresponding ordered partition set is denoted by $X = \{ x_1 , x_2 , ..., x_m \}$
\label{def:ordered_integer_parts}
\end{definition}

\begin{lemma}
Consider a tree $\mathcal{T}$ with a root node, $v_s \in \mathcal{V}$. Let $C$ be the set of children of the lowest linear descendant $\mathcal{L}(v_s)$ of the root node and let $P_C$ be the set of all partitions of $C$. Given an oracle $\mathcal{O}_m(\mathcal{T}_u)$ which computes all sweep-covers with $m < n$ subcovers on an induced subtree defined by $v_u \in \mathcal{V}$, we can compute all sweep-covers with $n$ subcovers in $\mathcal{T}$.
\label{lem:inductive_step}
\end{lemma}
\begin{proof}


First, every sequence of mappings must begin with repeated $\alpha$-mappings beginning with the root sweep-cover until it is replaced by the set of children $C$ of $\mathcal{L}(v_s)$. Then, we consider all possible ways of applying $\pi$-mappings to decompose $C$ so that we can apply an $\alpha$-mapping. By Lemma \ref{lem:all_sweep_construct_ap}, this decomposition is necessary to construct all of the remaining sweep-covers. So, applying every possible $\pi$-mapping is equivalent to considering every partition $p \in P_C$ and constructing all feasible sweep-covers using $p$ and the oracle. Each partition $p \in P_C$ is decomposable into two sets $L \subseteq p$ and $R \subseteq p$, where $L$ is the set of all singletons in $p$ and $R$ is the set of all non-singletons in $p$. For any partition $p \in P_C$, if $|p| = n$, then $p$ is a sweep-cover with $n$ subcovers. For any partition $|p| > n$, because $\alpha$-mappings maintain the number of subcovers and $\pi$-mappings only increase the number of subcovers, we cannot construct a sweep-cover of size $n$ using any $p$ where $|p| > n$ without mapping $p$ to a different partition. Because we will already consider all partitions, we will ignore any $p$ where $|p| > n$ to avoid double counting. For any partition $p$ such that $|p| < n$, there are two cases to consider. If $|L| == 0$, then we cannot apply an $\alpha$-mapping to the sweep-cover $S = \{p\}$, without changing $p$. So, for the same reasons as $|p| > n$, we ignore any partition where $p < n$ and $|L| == 0$. However, if $|L| > 0$, then we can apply $\alpha$-mappings to any of the singletons in $L$. This means that any subcovers after applying the $\alpha$-mappings will be in the induced subtrees defined by the singletons in $L$, i.e., $\mathcal{L} = \{\mathcal{T}_{L_1},\mathcal{T}_{L_2},...,\mathcal{T}_{|L|}\}$. Lemma \ref{lem:induced_subtree_covered} implies that all of these subcovers must form a sweep-cover for each induced subtree in $\mathcal{L}$. Because we already applied a $\pi$-mapping on $C$, the number of subcovers required from each induced subtree in $\mathcal{L}$ must be less than $n$ in total, i.e., $m \leq n - |R|$. To exhaustively consider all ways of finding $n - |R|$ subcovers from $|L|$ distinct induced subtrees, consider the set of all integer compositions of $n - |R|$ into $L$ parts, i.e., $\Gamma_{n - |R|, |L|}$. For each integer composition $X \in \Gamma_{n - |R|, |L|}$, we pair part $X_k$ with tree $\mathcal{T}_{L_k}$ and use the oracle $\mathcal{O}_{X_k}(\mathcal{T}_{L_k})$ to compute the set $\mathcal{Q}^{X_k}_{L_k}$ of all sweep-covers in $\mathcal{T}_{L_k}$ with $X_k$ subcovers. After computing $\mathcal{Q}^{X_k}_{L_k}$ for a given subtree and integer composition, the cartesian product $\mathcal{Q}^{X_k}_{L_1} \times \mathcal{Q}^{X_k}_{L_2} \times ... \times \mathcal{Q}^{X_k}_{L_{|L|}} \times R$ yields every way of combining subcovers from each subtree to form a sweep-cover. Doing this for every partition $p \in P_C$ and integer composition $X \in \Gamma_{n - |R|,|L|}$ exhaustively yields every possible sweep-cover with $n$ subcovers.
\end{proof}

 In the next section we will provide an algorithm for computing sweep-covers on a given tree using Lemma \ref{lem:inductive_step}.

\section{Algorithm to Compute Sweep-Covers in a Tree}
\label{sec:Algorithm}


\subsection{Algorithm to Find All Sweep-Covers}

\begin{algorithm*}[t!]
	\caption{Compute all Sweep Covers with $n$ subcovers in a tree $\mathcal{T}$}
	\label{alg:Find_Sweep_Covers}
	\textbf{Input}: A tree $\mathcal{T}$ and the number of subcovers $n$ in each sweep-cover.\\
	\textbf{Output}: A collection of all sweep covers $\mathcal{Q}$ with $n$ subcovers in $\mathcal{T}$.

\begin{algorithmic}[1]
\Procedure{SweepCovers}{$\mathcal{T}$,$n$}

\State{$\mathcal{Q} \leftarrow \emptyset$}
\State{$v_s \leftarrow$ root of $\mathcal{T}$}
\State{$C \leftarrow$ set of children of last node in the linear path beginning at $v_s$}
\If{$(n == 1)$}
    \State{\textbf{W} $\leftarrow$ collection of singletons on the linear path from $v_s$ to $\mathcal{L}(v_s)$}
    \State{ \textbf{$\mathcal{Q} \leftarrow \{\{W \cup \{C\} \}\}$ }}
\Else
\State{$H \leftarrow$ all ways of partitioning $C$ into $\{1,2,...,\min\{n,|C|\}\}$ parts}
\For{each $H_i \in H$}
    \State{$L \leftarrow$ a collection of all singletons in $H_i$}
    \State{$R \leftarrow$ a collection of all non-singletons in $H_i$}
    \If{($|L| == 0$ and $|R| == n$)}
        \State{$\mathcal{Q} \leftarrow \mathcal{Q} \cup \{\{R\}\}$}
    \Else
        \For{each integer composition $X \in \Gamma^{n-|R|}_{|L|}$}
            \State{$\textbf{W} = \emptyset$}
            \For{$k = 1 \rightarrow|L|$ }
                \State{Let $\mathcal{T}_{L_k}$ be the induced subtree defined by $L_k \in L$}
                \State{\textbf{W} $=$ \textbf{W} $\times$ SweepCovers($\mathcal{T}_{L_k}$,${X}_{k}$) }
                    
            \EndFor
            \State{\textbf{W} $=$ \textbf{W} $\times \, \{R\}$}
            
            \For{all $S \in$ \textbf{W} : $|S| == n$}
                \State{$\mathcal{Q} \leftarrow \mathcal{Q} \cup \{\{S\}\} $}
            \EndFor
        \EndFor
    \EndIf
\EndFor
\EndIf

\State{\textbf{ Return} $\mathcal{Q}$}

\EndProcedure
\end{algorithmic}
\label{alg:sweep_covers}
\end{algorithm*}

In this section, we provide Algorithm \ref{alg:sweep_covers} to compute all sweep-covers with $n$ subcovers on a given tree.



\subsection{Correctness Proofs}

\begin{lemma}
Algorithm \ref{alg:Find_Sweep_Covers} returns all possible sweep-covers with one subcover for any tree $\mathcal{T}$.
\label{lem:correctness:base_case}
\end{lemma}
\begin{proof}
The root sweep-cover can be considered a sweep-cover with one subcover. By applying $\alpha$-mappings beginning with the root sweep-cover $S=\{\{v_s\}\}$, eventually one recovers all sweep-covers with one subcover along the linear path beginning at $v_s$ plus the sweep-cover $S=\{C\}$. This is captured by lines 6 and 7. We can show these are the only sweep-covers with one subcover by using the properties of $\alpha$ and $\pi$ mappings. Beginning with the root sweep-cover, because there exists a linear path between the root node and $\mathcal{L}(v_s)$, the only allowed mappings are $\alpha$ mappings until the sweep-cover is $S=\{C\}$. However, an $\alpha$ mapping is no longer valid because $C$ is the only subcover and its non-singleton. So, the following mapping in the sequence must be a $\pi$ mapping. So, all constructible sweep-covers in the tree must begin with a sequence of $\alpha$-mappings followed by one $\pi$-mapping. Since $\pi$-mappings create more subcovers while $\alpha$-mappings maintain the number of subcovers, one cannot decrease the number of subcovers to 1 after the first use of the $\pi$-mapping. Therefore, all other sweep-covers will have at least two subcovers.
\end{proof}

\begin{theorem}
Algorithm \ref{alg:Find_Sweep_Covers} returns all sweep covers with $n$ subcovers from $\mathcal{T}$. If none are possible, then the algorithm returns an empty set.
\label{thm:correctness:full_correctness}
\end{theorem}

\begin{proof}
We will prove this by induction. Lemma \ref{lem:correctness:base_case} proved the base case where $n=1$. One can see that the algorithm adheres to the base case in lines 5-7. Now, assume for the inductive step that the algorithm returns all sweep-covers with fewer than $n$ subcovers in a tree $\mathcal{T}$. By applying the same arguments as in Lemma \ref{lem:inductive_step}, the algorithm itself acts as the oracle $\mathcal{O}_m(\mathcal{T})$ in Lemma \ref{lem:inductive_step} to verify the inductive step in lines 8-23. Lines 22-24 show that if it is impossible to find the correct number of subcovers in any given subtree, then $\mathcal{Q}$ will not contain a sweep-cover with fewer than $n$ subcovers. This is easily verifiable by induction too. If none are of size $n$, then $\mathcal{Q}$ is returned empty. Therefore, Algorithm \ref{alg:Find_Sweep_Covers} returns all sweep covers of size $n$ or an empty set.
\end{proof}

\section{Sweep-Covers in Infinite Trees}
 One important caveat differentiating the analysis of sweep-covers is that the enumeration expression for all sweep-covers can only be defined for a given class of trees. For the sake of providing a useful bound for algorithm analysis, in this section, we prove the recurrence relation for computing sweep-covers in a class of infinite trees with bounded out-degree and linear path length. The number of sweep-covers in this class of trees may be an upper bound for every other infinite class of trees, though no proof is provided in this work.

\begin{definition}[Infinite, Labeled, Directed (ILD) Trees]
Consider a class of infinite height, labeled, and directed trees (ILD trees) composed of only $\Delta$-star nodes and nodes with only one outgoing edge. Let there be a linear path of constant edge length, $\gamma$, between the children of any given $\Delta$-star and the next $\Delta$-star node. These trees are denoted as $\mathcal{T}_{\Delta,\gamma}$. We denote $P_{\Delta,\gamma}(n)$ to be the enumeration of all sweep-covers with $n$ subcovers on $\mathcal{T}_{\Delta,\gamma}$.
\end{definition}
\begin{figure}
    \centering
    \includegraphics[width=0.5\textwidth]{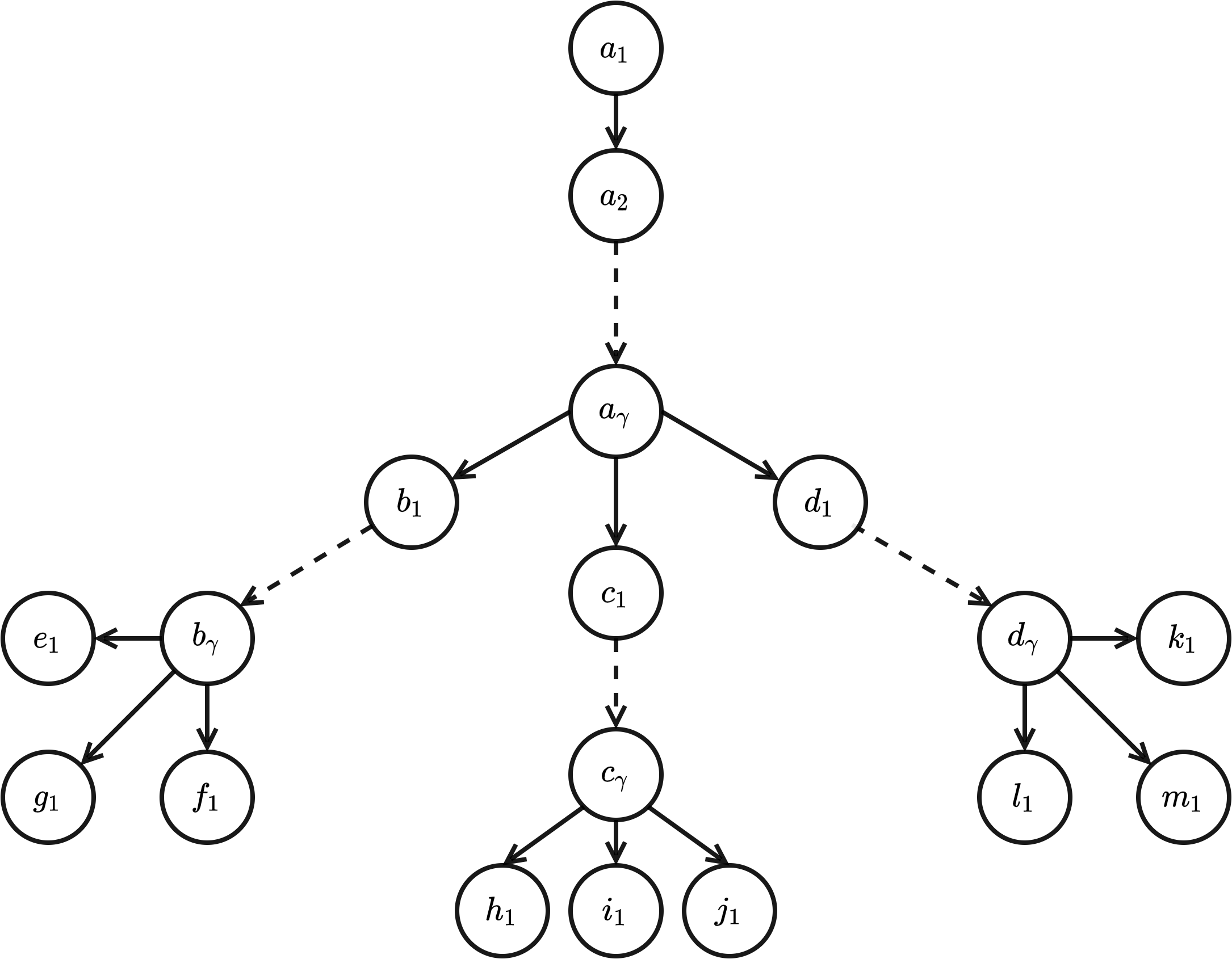}
    \caption{The infinite height, labeled, and directed tree for $\Delta = 3$, i.e., $\mathcal{T}_{3,\gamma}$}.
\label{fig:Partition_Tree}
\end{figure}

As we will show, the enumeration expression for the number of sweep-covers depends only on the number of subcovers in the sweep-cover, $n$, the bounded out-degree of the ILD tree, $\Delta$, and the maximum linear path length, $\gamma$. If $\gamma$ and $\Delta$ are fixed constants, then the number of sweep-covers of size $n$ is bounded even on infinite trees. This is because Definitions \ref{def:SweepCover}.\ref{sweep_cover:complete_graph} and \ref{def:SweepCover}.\ref{sweep_cover:sibling} force a $\pi$-mapping every time we wish to use an $\alpha$-mapping to construct a sweep-cover with nodes of greater depth. Therefore, the maximum depth of any node in a subcover is always bounded by $n$. If we take away either of these definition requirements, then the number of sweep-covers on infinite trees is also infinite. To define the enumeration expression for sweep-covers in ILD trees, we simply compute the enumeration expressions for Lemmas \ref{lem:inductive_step} and \ref{lem:correctness:base_case} when we assume the symmetry and path lengths of the ILD trees.

\begin{definition}
Define $R(n,m)$ to be the enumeration of all ways of placing $n$ unique elements into $m$ indistinct, non-singleton sets 
\begin{equation}
    R(n,m) \triangleq \sum_{s=n-m}^{n} \binom{n}{s} (-1)^{n-s} \stirling{s}{ s+m-n}  
\end{equation} where $\stirling{n}{k}$ is the Stirling number of the second kind \cite{boyadzhiev2018close,BONA2016500,knuth1992notes}. \label{def:R_nm_stirling}
\end{definition}

Given a complete set of sibling nodes $C$, define a subset $C' \subseteq C$ such that $|C'| = n$. Then, $R(n,m)$ is equivalent to enumerating all ways of constructing the collection $R$ with $m$ sets in Lemma \ref{lem:inductive_step} using only the nodes in $C'$. In other words, for each way of constructing $L$ such that $n$ nodes remain for $R$, $R(n,m)$ computes all ways of constructing $R$ with $m$ subcovers from the $n$ available nodes. Now, given a set of $n$ nodes for constructing $L$, we will enumerate all ways of constructing $k$ subcovers from the induced subtrees defined by the $n$ nodes.

\begin{definition}
Consider the tree $\mathcal{T}_{\Delta,\gamma}$ and its source node $v_s \in \mathcal{T}_{\Delta,\gamma}$. Let $C$ be the set of children of $\mathcal{L}(v_s)$. Let $p \in P_C$ be some partition of $C$ into a collection of singletons $L$ such that $|L| == k$. The following is the enumeration expression for all possible ways of combining a collection of sweep-covers from the induced subtrees defined by the nodes in $L$ such that there are $n$ subcovers in the collection. Note that due to the symmetry of ILD trees, all of the induced subtrees are identical to the same ILD tree. So, the number of sweep-covers in each subtree reduces to $P_{\Delta,\gamma}(n)$.

\label{def:Ldelta}
\begin{equation}
    L_{\Delta}(n,k) \triangleq \sum_{X \in \Gamma_{n,k} }{\prod_{x_j \in X}P_{\Delta,\gamma}(x_j)} 
    \label{eq:LDelta}
\end{equation} 
\end{definition}

\begin{lemma}
 The following expression is the recurrence relation for enumerating sweep-covers on ILD trees $\mathcal{T}_{\Delta,\gamma}$ for any out-degree $\Delta > 1$. Note that for intuition, $l$ and $r$ iterate through the sizes of $L$ and $R$, respectively, from Lemma \ref{lem:inductive_step}.
 \begin{equation}
      P_{\Delta,\gamma}(1) = \gamma + 1 \label{eq:recurrence1}
 \end{equation}
\begin{align}
P_{\Delta,\gamma}(n) =& 
\sum_{l=1}^{\Delta - 2}{ \binom{\Delta}{l}
\sum_{r=1}^{\Delta - l} { L_{\Delta}(n-r,l) 
R(\Delta - l,r)}} \label{eq:Precurr_left} \\
&+ L_{\Delta}(n,\Delta) \label{eq:Precurr_mid} \\
&+ R(\Delta,n) \label{eq:Precurr_right}
\end{align}\label{lem:recurrence}
\end{lemma}
\begin{proof} 
 The base case for $n = 1$ is given by $P_{\Delta,\gamma}(1)$ and proven by counting the number of $\alpha$-mappings using Lemma \ref{lem:correctness:base_case}. For $n > 1$, consider the recursive definition of a sweep-cover given by Lemma \ref{lem:inductive_step} and the set of children $C$ of $\mathcal{L}(v_s)$. By Definitions \ref{def:R_nm_stirling} and \ref{def:Ldelta}, given a collection $L$ of $l$ singletons and a collection $R$ of $r$ non-singletons from a partition $p \in P_C$, $L_\Delta(n - r,l)R(\Delta-l,r)$ enumerates all ways of constructing a sweep-cover with $n$ subcovers. Due to the symmetry of $\mathcal{T}_{\Delta,\gamma}$, the number of sweep-covers from the induced subtrees defined by $L$ for a given $l$ is equal for each induced subtree. So, for each way of choosing $l$ singletons from $\Delta$ nodes, the number of sweep-covers is the same, hence we can multiply by $\binom{\Delta}{l}$ to enumerate each way of choosing $l$ singletons. Therefore, $\binom{\Delta}{l}L_\Delta(n - r,l)R(\Delta-l,r)$ enumerates all ways of applying Lemma \ref{lem:inductive_step} for a fixed $|R| = r$ and $|L| = l$.  Likewise, $R(\Delta-l,r)$ computes all ways of distributing the remaining $\Delta - l$ nodes into $r$ non-singleton sets. Because both $|L| = \Delta - 1$ and $|L| = \Delta$ imply $\Delta$ sets of singletons, we separate $|L| = 1 \rightarrow \Delta - 2$ and $|L| = \Delta$ for the enumeration, hence (\ref{eq:Precurr_mid}). Likewise, due to $l = 0$ implying $L_{\Delta}(n-r,l) = 0$, we also separate the term that computes all ways of constructing a sweep-cover with $n$ subcovers using $C$, i.e., $R(\Delta,n)$.
\end{proof}

\section{Raney Numbers as Lower Bounds for ILD Trees}
\label{sec:Raney}

\begin{definition}[Raney Numbers]
Raney numbers are a two-parameter generalization of Catalan numbers introduced by Raney, who was studying functional composition patterns \cite{Aval2008,raney1960functional}.

\begin{equation}
    C_{p,r}(k) \triangleq \frac{r}{kp+r}\binom{kp+r}{k}
\end{equation}

\end{definition}

\begin{definition}[Raney Tree]
A Raney Tree of type (p,r,k), also referred to as a coral diagram of type (p,r,k) \cite{ZHOU2017114}, is defined to be a labeled tree whose root node has an out-degree of $r$ and all other $k - 1$ internal nodes have an out-degree of $p$. 
\end{definition}
The enumeration of all labeled Raney Trees of a given type was proven to be equal to its Raney number, $C_{p,r}(k)$, by Zhou and Yan \cite{ZHOU2017114}.

\begin{corollary}
 $P_{2,0}(n) = C_{2,1}(n+1)$. In other words, the number of sweep-covers of size n on an ILD of $\Delta = 2$ and $\gamma = 0$ is equal to the Catalan number for n+1. \label{cor:Catalan_Proof}
\end{corollary}
\begin{proof}
First note that $P_{2,0}(n) =  L_2(n,2) + R(2,n)$ and for $n = 1$, $P_{2,0}(1) = 1$. Using Segner's recurrence relation for $n > 1$ \cite{Segner1758}: 
\begin{align}
    P_{2,0}(n) = L_{\Delta,\gamma}(2,n) = \sum_{n' = 1}^{n}P_{2,0}(n-n')P_{2,0}(n') = C_{2,1}(n+1).
\end{align} \end{proof}
Continuing this observation, we can write the following recurrence relation for Raney trees in a similar manner to the recurrence relation for sweep-covers.
\begin{theorem}
 \begin{equation}
    C_{p,r}(k) = \sum_{l=1}^{r} \binom{r}{l}\sum_{X \in \Gamma_{k-1,l}} \prod_{x_j \in X} C_{p,p}(x_j)
\end{equation}
\end{theorem}
\begin{proof}
Beginning with an $r$-star root node, we can construct all Raney trees of coral type $(p,r,k)$ by enumerating all ways of choosing the leaves of the $r$-star and constructing trees with $k - 1$, $p$-star internal nodes rooted at the leaves. We can do this exhaustively by considering all ways of choosing a subset of the leaves of the $r$-star. For each subset of leaves with size $l$, we consider all integer compositions of $k-1$ internal nodes into the $l$ available leaves. For each integer composition $X \in \Gamma_{k-1,l}$, we assign one part to one leaf node and compute the number of trees with $x_j$, $p$-star internal nodes, i.e., $C_{p,p}(x_j)$. 
\end{proof}

\begin{lemma}
 For each Raney tree of type $(p,p,k)$, there exists a sweep-cover in $\mathcal{T}_{p,0}$ with $k$ subcovers, i.e., $P_{\Delta,0}(n) \geq C_{\Delta,\Delta}(n)$. \label{lem:Raney_Tree_to_ILD}
\end{lemma}
\begin{proof}
First note that because $\gamma = 0$ and $\mathcal{T}_{p,0}$ is infinite, we can construct a subtree $\mathcal{R} \subset \mathcal{T}_{p,0}$ for each way of choosing $k$ connected $p$-star nodes beginning at the root node, i.e., all Raney trees of type $(p,p,k)$. Note that we can enumerate all ways of choosing $k$ internal nodes of $\mathcal{T}_{p,0}$ by following $\alpha$ and $\pi$ mappings and choosing the singleton node as the internal node of the subtree whenever we use an $\alpha$-mapping. Then, because we can construct all Raney trees using $\alpha$ and $\pi$ mappings, a sweep-cover with $k$ subcovers must exist for each Raney tree with $k$ internal nodes.
\end{proof}
\begin{lemma} \label{thm:P_Raney}
 \begin{equation}
    P_{\Delta,\gamma}(n) \geq C_{\Delta,\Delta}(n)\gamma
\end{equation}
\end{lemma}
\begin{proof}
By increasing $\gamma$ in any ILD tree $\mathcal{T}_{p,\gamma}$, only the linear path length in each embedding tree for $\mathcal{T}_{p,\gamma}$ will increase. Therefore, for each sweep-cover in $\mathcal{T}_{p,\gamma}$ with $k$ subcovers, there exists a sweep cover in $\mathcal{T}_{p,\gamma+1}$ with $k$ subcovers. For each embedding tree, there exists at least one internal node at the lowest depth. By applying $\alpha^{-1}$-mappings to this node, we can construct $\gamma+1$ sweep-covers which are heteromorphic to every other sweep-cover.
\end{proof}
\begin{theorem}
 \begin{equation}
     P_{\Delta,\gamma}(n) = \Omega\left( \frac{\sqrt{2 \pi} n^{\Delta n + \Delta + \frac{3}{2}}}{e^n ((\Delta-1)n+\Delta+1)!(n+1)!} \gamma \right)
 \end{equation}
\end{theorem}
\begin{proof}
Using Lemma \ref{thm:P_Raney}, we compute the asymptotics 
\begin{align}
    C_{\Delta,1}(n+1) &= \frac{1}{(n+1)\Delta+1}\binom{(n+1)\Delta+1}{ n + 1} = \frac{1}{(n+1)!} \prod_{i=0}^{n-1}(\Delta(n+1) + 1 - i))\\
    &=\frac{1}{(n+1)!}((\Delta-1)n+\Delta+2)^{(n)}
\end{align}
where $((\Delta-1)n+\Delta+2)^{(n)}$ is the Pochhammer symbol which satisfies the lower bound \cite{abramowitz+stegun}.
\end{proof}

\section{Sweep-Cover Sequences}
\label{sec:Data}
In this section, we provide a table of computed values for enumerating sweep-covers on ILD trees. This table assumes $\gamma = 0$. Python code is available for generating these integers and they've been hand checked by following Lemma \ref{lem:inductive_step} for $\Delta \le 5$ and $n \le 6$ \cite{blake_wilson_2020_4034921}.


\begin{table}
\centering
\resizebox{\textwidth}{!}{
\begin{tabular}{lllllllll}
\hline
 \multicolumn{9}{|c|}{$P_{\Delta,0}(n)$} \\ \hline
\multicolumn{1}{|l|}{ $\Delta \backslash n$}\vrule width 1pt &
  \multicolumn{1}{l|}{1} &
  \multicolumn{1}{l|}{2} &
  \multicolumn{1}{l|}{3} &
  \multicolumn{1}{l|}{4} &
  \multicolumn{1}{l|}{5} &
  \multicolumn{1}{l|}{6} &
  \multicolumn{1}{l|}{7} &
  \multicolumn{1}{l|}{8} \\ \hlineB{2}
\multicolumn{1}{|l|}{2}\vrule width 1pt &
  \multicolumn{1}{l|}{1} &
  \multicolumn{1}{l|}{1} &
  \multicolumn{1}{l|}{2} &
  \multicolumn{1}{l|}{5} &
  \multicolumn{1}{l|}{14} &
  \multicolumn{1}{l|}{42} &
  \multicolumn{1}{l|}{132} &
  \multicolumn{1}{l|}{429} \\ \hline
\multicolumn{1}{|l|}{3}\vrule width 1pt &
  \multicolumn{1}{l|}{1} &
  \multicolumn{1}{l|}{3} &
  \multicolumn{1}{l|}{10} &
  \multicolumn{1}{l|}{39} &
  \multicolumn{1}{l|}{174} &
  \multicolumn{1}{l|}{846} &
  \multicolumn{1}{l|}{4332} &
  \multicolumn{1}{l|}{22959} \\ \hline
\multicolumn{1}{|l|}{4}\vrule width 1pt &
  \multicolumn{1}{l|}{1} &
  \multicolumn{1}{l|}{7} &
  \multicolumn{1}{l|}{34} &
  \multicolumn{1}{l|}{221} &
  \multicolumn{1}{l|}{1614} &
  \multicolumn{1}{l|}{12394} &
  \multicolumn{1}{l|}{99556} &
  \multicolumn{1}{l|}{827045} \\ \hline
\multicolumn{1}{|l|}{5}\vrule width 1pt &
  \multicolumn{1}{l|}{1} &
  \multicolumn{1}{l|}{15} &
  \multicolumn{1}{l|}{100} &
  \multicolumn{1}{l|}{1035} &
  \multicolumn{1}{l|}{11376} &
  \multicolumn{1}{l|}{132930} &
  \multicolumn{1}{l|}{1630860} &
  \multicolumn{1}{l|}{20606355} \\ \hline
\multicolumn{1}{|l|}{6}\vrule width 1pt &
  \multicolumn{1}{l|}{1} &
  \multicolumn{1}{l|}{31} &
  \multicolumn{1}{l|}{276} &
  \multicolumn{1}{l|}{4511} &
  \multicolumn{1}{l|}{70986} &
  \multicolumn{1}{l|}{1232752} &
  \multicolumn{1}{l|}{22295588} &
  \multicolumn{1}{l|}{4.16E+08} \\ \hline
\multicolumn{1}{|l|}{7}\vrule width 1pt &
  \multicolumn{1}{l|}{1} &
  \multicolumn{1}{l|}{63} &
  \multicolumn{1}{l|}{742} &
  \multicolumn{1}{l|}{19215} &
  \multicolumn{1}{l|}{418698} &
  \multicolumn{1}{l|}{10810254} &
  \multicolumn{1}{l|}{2.82E+08} &
  \multicolumn{1}{l|}{7.65E+09} \\ \hline
\multicolumn{1}{|l|}{8}\vrule width 1pt &
  \multicolumn{1}{l|}{1} &
  \multicolumn{1}{l|}{127} &
  \multicolumn{1}{l|}{1982} &
  \multicolumn{1}{l|}{81565} &
  \multicolumn{1}{l|}{2409926} &
  \multicolumn{1}{l|}{93612646} &
  \multicolumn{1}{l|}{3.45E+09} &
  \multicolumn{1}{l|}{1.37E+11} \\ \hline
\multicolumn{1}{|l|}{9}\vrule width 1pt &
  \multicolumn{1}{l|}{1} &
  \multicolumn{1}{l|}{255} &
  \multicolumn{1}{l|}{5320} &
  \multicolumn{1}{l|}{347115} &
  \multicolumn{1}{l|}{13769616} &
  \multicolumn{1}{l|}{8.16E+08} &
  \multicolumn{1}{l|}{4.18E+10} &
  \multicolumn{1}{l|}{2.45E+12} \\ \hline
 &
   &
   &
   &
   &
   &
   &
   &
  
\end{tabular}
}
\end{table}

\acks
This work was supported by the Open Access Funding Project at Purdue University.

\bibliographystyle{unsrt}
\bibliography{references}
\pagebreak
\appendix
\section{Full Sweep-Cover Example}
\label{app:sweep_cover}
\begin{figure}[h!]
    \centering
    \includegraphics[width=0.95\textwidth]{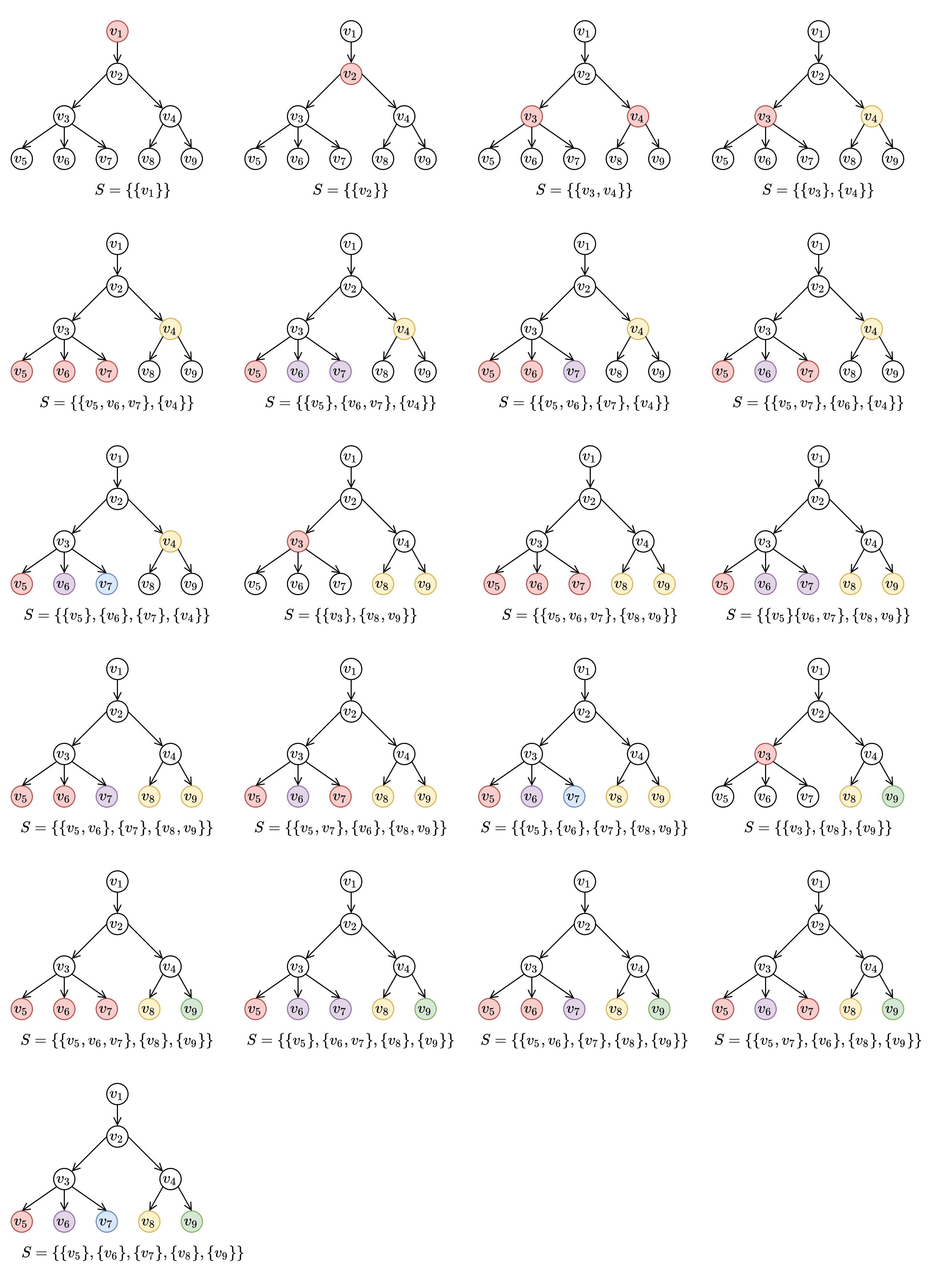}
    \caption{All sweep covers for a tree.}
    \label{fig:full_sweep_cover_example}
\end{figure}

\end{document}